\documentclass[oneside,english]{amsart}
\usepackage[T1]{fontenc}
\usepackage[latin9]{inputenc}
\usepackage{units}
\usepackage{amsmath}
\usepackage{amsthm}
\usepackage{amstext}
\usepackage{amssymb}
\usepackage{esint}
\usepackage[numbers]{natbib}
\usepackage{nicefrac}
\usepackage{xfrac}

\makeatletter
\numberwithin{equation}{section}
\numberwithin{figure}{section}
\theoremstyle{plain}
\newtheorem{thm}{\protect\theoremname}
  \theoremstyle{plain}
\newtheorem{cor}[thm]{Corollary}
  \theoremstyle{plain}
  \newtheorem{lem}[thm]{\protect\lemmaname}
  \theoremstyle{definition}
  \newtheorem{defn}[thm]{\protect\definitionname}
  \theoremstyle{remark}
  \newtheorem*{rem*}{\protect\remarkname}
  \theoremstyle{remark}
  \newtheorem{rem}[thm]{\protect\remarkname}
  \theoremstyle{plain}
  \newtheorem{prop}[thm]{\protect\propositionname}
  \theoremstyle{plain}
  \newtheorem*{thm*}{\protect\theoremname}
\numberwithin{thm}{section}
\makeatother

\usepackage{babel}
  \providecommand{\definitionname}{Definition}
  \providecommand{\lemmaname}{Lemma}
  \providecommand{\propositionname}{Proposition}
  \providecommand{\remarkname}{Remark}
  \providecommand{\theoremname}{Theorem}
\providecommand{\theoremname}{Theorem}

\begin{document}
\global\long\def\GG{\mathbb{G}}
\global\long\def\rank#1#2{\textrm{rank}_{#2}\left(#1\right)}
\global\long\def\malg#1{\mathcal{B}_{#1}}
\global\long\def\vnalg#1{\mathcal{A}_{#1}}
\global\long\def\salg#1{\Sigma_{#1}}
\global\long\def\SL#1{\mathrm{SL}_{#1}}
\global\long\def\GL#1{\mathrm{GL}_{#1}}
\global\long\def\SLk#1#2{\mathrm{SL}_{#1}\left(#2\right)}
\global\long\def\HH{\mathbb{H}}
\global\long\def\PP{\mathbb{P}}
\global\long\def\VV{\mathbb{V}}
\global\long\def\QQ{\mathbb{Q}}
\global\long\def\R{\mathbb{R}}
\global\long\def\C{\mathbb{C}}
\global\long\def\SS{\mathbb{S}}
\global\long\def\N{\mathbb{N}}
\global\long\def\Z{\mathbb{Z}}
\global\long\def\Qp{\mathbb{Q}_{p}}
\global\long\def\SubG#1{\textrm{Sub}(#1)}

\title{The Nevo-Zimmer intermediate factor theorem over local fields}

\author{Arie Levit}

\maketitle

\begin{abstract}

The Nevo-Zimmer  theorem classifies the possible intermediate $G$-factors $Y$ in $X \times \nicefrac{G}{P} \to Y \to X$, where $G$ is a higher rank semisimple Lie group, $P$  a minimal parabolic and $X$  an irreducible $G$-space with an invariant probability measure.

An important corollary is the Stuck-Zimmer theorem, which states that a faithful irreducible action of a higher rank Kazhdan semisimple Lie group with an invariant probability measure is either transitive or free, up to a null set.

We present a different proof of the first theorem, that allows us to extend these two well-known theorems to linear groups over arbitrary local fields.
\end{abstract}

\section{Introduction}
\label{sec:introduction}

Let $k$ be a local field and let $\GG$ be a connected simply-connected semisimple
algebraic $k$-group without $k$-anisotropic almost $k$-simple subgroups\footnote{Subsection \ref{sub:subgroup_structure} contains definitions concerning algebraic $k$-groups, as well as some examples.}.

\subsection{The Nevo-Zimmer intermediate factor theorem}
\label{sub:intro_intermediate_factor_theorem}

Our central task is to prove the following  version of the Nevo-Zimmer intermediate factor theorem for linear algebraic groups local fields.
\begin{thm}
\textbf{\label{thm:Intermediate-factor-theorem}} Assume that $\rank{\GG}k\ge 2$ and let $\PP$
be a minimal parabolic $k$-subgroup. Let $X$ be an irreducible
 $\GG_{k}$-space 
  with finite invariant measure.  Given an ergodic $\GG_{k}$-space $Y$ and a pair of $\GG_k$-maps \footnote{The reader is referred to Subsection \ref{sub:measurable_group_actions} for a discussion of measurable group actions, including the definitions of  a $G$-space and a $G$-map.}
\[
X\times\nicefrac{\GG_{k}}{\PP_{k}}\to Y\to X
\]
whose composition is the projection, there is a parabolic $k$-subgroup
$\QQ$ containing $\PP$ such that 
\[
Y\cong X\times\nicefrac{\GG_{k}}{\QQ_{k}}
\]
as a $\GG_{k}$-space. Moreover the maps $X\times\nicefrac{\GG_{k}}{\PP_{k}}\to Y$
and $Y\to X$ are projections.
\end{thm}


The case $k=\R$ of Theorem \ref{thm:Intermediate-factor-theorem} is essentially the intermediate factor theorem for real Lie groups, which first appeared in Zimmer's work \cite{zimmer1982ergodic}. The proof given in \cite{zimmer1982ergodic} is a natural generalization of Margulis' proof of the factor theorem \cite{margulis1978quotient}. However, as was pointed out by Nevo and Zimmer in Section 4 of \citep{nevo2002generalization} that proof contains an imprecise argument.

An additional treatment of the real Lie group case is given in a series of papers by Nevo and Zimmer (\citep{nevo1999homogenous, nevo2002generalization}). In fact, a much stronger version called the generalized intermediate factor theorem is proved in \citep{nevo2002generalization}. This approach uses ideas of Furstenberg such as the interplay between stationary and $P$-invariant measures, and diverges from the original approach of Margulis in \cite{margulis1978quotient}.

We mention that a corrected proof along the lines of \cite{zimmer1982ergodic} was given by Nevo and Zimmer in the unpublished work \cite{nevo1998proof}.

Our chief motivation is to present a proof of Theorem \ref{thm:Intermediate-factor-theorem} that is as close as possible to the lines of \cite{zimmer1982ergodic} and \cite{margulis1978quotient}, while extending the result to local fields. Indeed, the factor theorem \cite{margulis1978quotient} is already given in this (indeed, even greater) generality. We remark that our approach differs from that of \cite{nevo1998proof}. 

Indeed, the real case of Theorem \ref{thm:Intermediate-factor-theorem} implies the real Lie group version at least for center-free semisimple Lie groups, and the same remark applies to Theorem \ref{thm:stuck-zimmer-over-local-fields}. This reduction is based on the fact that every connected semisimple real Lie group with trivial center is isomorphic to $\GG_{\R}^{0}$
for some connected but possibly non simply-connected algebraic $\R$-group. See Subsection \ref{sub:the real case} for further discussion.

An important corollary of the intermediate factor theorem is the theorem of Stuck and Zimmer discussed in Subsection \ref{sub:intro_nevo_stuck_zimmer_theorem} below. Moreover the factor theorem is a consequence of the intermediate factor theorem,
and the normal subgroup theorem for groups having property $\left(T\right)$
is a consequence of the Stuck-Zimmer theorem. So these four results are related, implying and generalizing each other.

We mention that the Nevo-Zimmer intermediate factor theorem has been extended in several directions since it first appeared; see for example the works of Bader-Shalom \cite{bader2006factor}, Creutz-Peterson \cite{creutz2013stabilizers} and the above mentioned \cite{nevo2002generalization}.

\subsection{The Stuck-Zimmer theorem}
\label{sub:intro_nevo_stuck_zimmer_theorem}

The following is a formulation of this theorem in the setting of linear algebraic groups over local fields.

\begin{thm}
\label{thm:stuck-zimmer-over-local-fields} 
Assume that $\rank{\GG}k \ge 2$ and $\GG_{k}$ has property
(T). Then every faithful, properly ergodic, irreducible
and finite measure preserving action of $\GG_{k}$ is essentially
free.

More generally, if the action has central kernel (but is possibly not faithful)  then $\mu$-almost every point has central stabilizer.
\end{thm}


For example as $\GG=\textrm{SL}_n$ is simply-connected the above theorem holds for the groups $\textrm{SL}_n(\R)$, $\textrm{SL}_n(\Qp)$ and $\textrm{SL}_n \left( \mathbb{F}_q ((t)) \right)$ for every prime $p$ and prime power $q$ as long as $n\ge3$. 

The classical theorem of Stuck and Zimmer \citep{stuck1994stabilizers} is the analogous statement for actions of semisimple real Lie groups of real rank at least $2$ with property $(T)$, having finite center and no compact factors.

The proof given in  \citep{stuck1994stabilizers} is deduced in a quite general fashion from the Nevo-Zimmer intermediate factor theorem. In fact, the possibility of the current generalization is already suggested there (see Section 5.3 of \citep{stuck1994stabilizers}). Therefore our contribution to Theorem \ref{thm:stuck-zimmer-over-local-fields} is little more than its restatement in the above form, and the application of Theorem \ref{thm:Intermediate-factor-theorem}.

Let us note that the Stuck-Zimmer theorem has been recently extended in the above mentioned work \cite{creutz2013stabilizers} as well as in \cite{hartman2013stabilizer}.

\begin{rem*}
$\GG_{k}$ has property (T) whenever $\rank{\HH}k \neq 1$  for every almost $k$-simple subgroup
$\HH$ of $\GG$, and this is in fact a necessary condition unless $k$ is $\R$ or $\C$. 
\end{rem*}

\subsection{Rigidity of invariant random subgroups}
\label{sub:applications of the stuck-zimmer theorem}

A context in which Theorem \ref{thm:stuck-zimmer-over-local-fields} admits an
interesting and powerful application is that of invariant random subgroups. 

Given a second countable locally compact group $G$, an \emph{invariant random subgroup} of $G$ is a Borel probability measure on the space of closed subgroups of $G$ that is invariant under conjugation.

The space of closed subgroups is regarded with the Chaubuty topology. For more information on this topology and on invariant random subgroups the reader is referred to e.g. \cite{abert2012growth, abert2012kesten,gelander2015lecture}. See also Remark \ref{rem:on chaubuty and effros borel} below.

Obvious examples of invariant random subgroups include Dirac measures $\delta_{N}$ supported on normal subgroups $N\vartriangleleft G$. Moreover given a lattice $\Gamma \le G$ there is an invariant random subgroup $\mu_{\Gamma}$ obtained as the push-forward of the probability measure on $G / \Gamma$ to the space of closed subgroups via the conjugation map $g\Gamma \mapsto g \Gamma g^{-1}$. In this sense invariant random subgroups are a common generalization of normal subgroups and lattices.

The Stuck-Zimmer theorem is used in \cite{abert2012growth} to show that the above examples are the only possible ones for a higher rank real Lie group. In light of Theorem \ref{thm:stuck-zimmer-over-local-fields} we are able to extend this rigidity phenomenon to local fields and deduce

\begin{cor}
	\label{cor:On IRS}
	Let $\GG$ be an algebraic  $k$-group as above. In particular assume that $k\textrm{-rank}(\GG)\ge 2$ and that $\GG_k$ has property $(T)$. Then every non-atomic irreducible invariant random subgroup of $\GG_k$ is of the form $\mu_{\Gamma}$ for some irreducible lattice $\Gamma\le G$.
\end{cor}

Indeed, a more general classification result is proved in \cite{abert2012growth} for reducible invariant random subgroups in real Lie groups, relying on the reducible case of the Stuck-Zimmer theorem, and is applied towards a study of asymptotic geometry of locally symmetric spaces. In \cite{gelevit} with Gelander we are able to extend some of the results of \cite{abert2012growth} concerning invariant random subgroups  to the non-Archimedean setting relying on the results of the present paper and in particular on Corollary \ref{cor:On IRS}.

We remark that contrary to Corollary \ref{cor:On IRS}, rank one groups admit a great variety of wild invariant random subgroups, see e.g. \cite{abert2012growth,bowen2012invariant}.

\subsection*{Proof strategy}

In the attempt to overcome the problem that was encountered in \cite{zimmer1982ergodic} we focus our attention on a certain Main Lemma (see Subsection \ref{sub:the main lemma}) which is a  key ingredient in the proof of Theorem \ref{thm:Intermediate-factor-theorem}. 

Measures algebras are a key player in the proof of the Main Lemma and are discussed in detail in Subsection \ref{sub:measure_algebras}. We regard a measure algebra as a topological space with the convergence in measure topology, see Subsection \ref{sub:convergence in measure}. It turns out that a sub-algebra is closed in this topology.

We study the measure algebra $\malg{Y}$ of the intermediate factor $Y$ and decompose it as a direct integral $\malg{Y}(x)$ over $x \in X$ where every $\malg{Y}(x)$ is a measure sub-algebra of $\malg{G/P}$. We develop a machinery that allows us to claim, given elements $C_n \in \malg{Y}(x_n)$ tending to $C \in \malg{Z}$ in the convergence in measure topology and $x_n \in X$ "tending" to $x \in X$ in a certain special sense, that $C \in \malg{Y}(x)$. This is done by introducing in Subsection \ref{sub:Effros-Borel-space} the Effros Borel structure on the space of all measure sub-algebras of $\malg{Y}$. We refrain however from defining a topology corresponding to this "convergence" of $x_n$ towards $x$; see Lemma \ref{lem:Effros convergence weaker form} for the exact statement.

The remainder of the proof of Theorem \ref{thm:Intermediate-factor-theorem} as well as the proof of Theorem \ref{thm:stuck-zimmer-over-local-fields} are essentially the same as in the original treatments and are only recalled here for the reader's convenience.

Finally in Section \ref{sec:non simply connected and the real case} we treat the case where the group $\GG$ is not simply-connected and prove Theorems \ref{thm:IFT-general form} and \ref{thm:SZ - general form} which are the generalizations of our two main results to this case. A certain subgroup $\GG_k^+$ plays an important role here and introduces several complications, see Subsection \ref{sub:the group G+} for more details. While this last section can be viewed as supplementary, it is required to deduce the real Lie group variants of our main results as in explained in Subsection \ref{sub:the real case}.

\vline{}

Theorems \ref{thm:Intermediate-factor-theorem} and \ref{thm:stuck-zimmer-over-local-fields} are proved in Sections \ref{sub:proof of ift} and \ref{sec:The-Stuck-Zimmer-theorem}, respectively. 

Invariant random subgroup rigidity and Corollary \ref{cor:On IRS} are discussed in Section \ref{sec:invariant random subgroup rigidity}.

\subsection*{Acknowledgments} I would like to thank Uri Bader, Tsachik Gelander, Shahar Mozes and Amos Nevo for many helpful discussions. I would like to thank the anonymous referee for his careful reading of this paper as well as for  his many suggestions and comments that without doubt were a major contribution in improving upon an earlier version. 

\section{$k$-algebraic groups and their actions}
\label{sec:k-groups-and-their-actions}

We recall for the reader's convenience some facts and notations regarding probability measure preserving group actions as well as the structure theory of $k$-algebraic groups.

In what follows $k$ denotes an arbitrary local field.

\subsection{Measurable group actions}
\label{sub:measurable_group_actions}

Let $G$ be second countable locally compact group. 
A \emph{$G$-space} is a standard
$\sigma$-finite Borel measure space $\left(X,\mu\right)$ with a
Borel measurable action $G\times X\to X$ such that $\mu$ is quasi-invariant.
A \emph{$G$-map} is a Borel map $f$ of $G$-spaces $f : (X,\mu)  \to (Y,\eta)$ that is $G$-equivariant and  satisfies $f_* \mu = \eta$.

Consider a given $G$-space $X$ and the associated $G$-action. 
The action is
\emph{ergodic} if every measurable $G$-invariant subset of $X$ is
either null or conull. It is \emph{essentially transitive} is there
exists a conull orbit. Clearly an essentially transitive action is
ergodic. An action is \emph{properly ergodic} if it is ergodic and not
essentially transitive, or equivalently every $G$-orbit has measure
$0$. An ergodic action is \emph{irreducible} if every non-central
normal subgroup acts ergodically. The action is \emph{faithful} if
for every non-trivial element $g\in G$, $gx\neq x$ holds for $x\in X$ of positive measure.
Finally, the action is \emph{essentially free} if $\mu$-a.e. $x\in X$
has trivial stabilizer in $G$.

\subsection{Subgroup structure of semisimple $k$-groups}
\label{sub:subgroup_structure}

An algebraic group $\GG$ is said to be \textit{defined over $k$}, or simply a \textit{$k$-group}, if the underlying variety as well as the group multiplication and inverse maps are defined over $k$ (see 0.10 of \citep{margulis1991discrete} or \cite{borel1991linear} for details).

Consider a connected semisimple $k$-group $\GG$ 
and let $G=\GG_{k}$ denote its group of $k$-points. $G$ is naturally a locally compact second countable topological group (see 3.1 of \citep{platonov1994algebraic}) so that the definitions of Subsection \ref{sub:measurable_group_actions} are applicable to actions of $G$. In fact,  in  characteristic zero $G$  has in addition an analytic structure.

$\GG$ is said to be \textit{$k$-anisotropic} if $\rank{\GG}{k}=0$, and this is the case if and only if $G$ is compact (see 3.1 of \citep{platonov1994algebraic}). 
$\GG$ is said to be \textit{simply-connected} if every central isogeny from a connected algebraic group into $\GG$ is an algebraic group isomorphism. For example,  $\textrm{SL}_n$ is simply-connected.

We recall the terminology from \citep{margulis1991discrete,margulis1978quotient} regarding several subgroups of particular interest. 
Let $\SS$ be a maximal $k$-split torus in $\GG$ and $\PP$
be a minimal parabolic $k$-subgroup containing $\SS$. We choose
an ordering on the set of roots $\Phi=\Phi\left(\SS,\GG\right)$ which
corresponds to $\PP$ and denote by $\Delta$ the set of simple roots
with respect to this ordering. 

For every subset $\theta\subset\Delta$
we consider (see I.1.2 of \citep{margulis1991discrete}): 
\begin{itemize}
\item The torus $\SS_{\theta} \subset \SS$ which is the connected component of the
intersection of the kernels of all $\alpha\in\theta$.
\item The standard parabolic $k$-subgroup $\PP_{\theta}$ corresponding to
$\theta$ and defined by the condition that it contains $\PP$ and
that $Z_{\GG}\left(\SS_{\theta}\right)$ is the reductive Levi component
of $\PP_{\theta}$.
\item The opposite parabolic $\overline{\PP}_{\theta}$ such that $\PP_{\theta}\cap\overline{\PP}_{\theta}$
is the Levi $Z_{\GG}\left(\SS_{\theta}\right)$.
\item The unipotent radicals $\VV_{\theta}=R_{u}\left(\PP_{\theta}\right)$
and $\overline{\VV}_{\theta}=R_{u}\left(\overline{\PP}_{\theta}\right)$.
\end{itemize}
Note that $\PP=\PP_{\emptyset}$ and $\GG=\PP_{\Delta}$. In fact, the parabolic group $\PP_\theta$ is generated by its subgroups $\PP_b$ for $b \in \theta$.
We denote the corresponding groups of $k$-points
by
\[
P_{\theta}=\left(\PP_{\theta}\right)_{k},\,\overline{P}_{\theta}=\left(\overline{\PP}_{\theta}\right)_{k},\, S_{\theta}=\left(\SS_{\theta}\right)_{k},\, V_{\theta}=\left(\VV_{\theta}\right)_{k},\,\overline{V}_{\theta}=\left(\overline{\VV}_{\theta}\right)_{k}
\]
 groups without subscripts ($P$,$V$ etc.) will refer to $\theta=\emptyset$.
For each $\theta$ write $\overline{L}_{\theta}=P_{\theta}\cap\overline{V}$.
Then $\overline{V}$ decomposes as $\overline{V}=\overline{L}_{\theta}\ltimes\overline{V}_{\theta}$.

\begin{lem}
\label{lem:relations of V P and G}Let $\theta\subset\Delta$ be a collection of simple roots. Then
the map $\varepsilon:\overline{V}_{\theta}\times P_{\theta}\to G$
sending $\left(v,p\right)\in\overline{V}_{\theta}\times P_{\theta}$
to $vp^{-1}\in G$ is a homeomorphism onto its image $\overline{V}_{\theta}\cdot P_{\theta}$,
which is open in $G$. Moreover $\overline{V}_{\theta}\cdot P_{\theta}$
is conull in $G$ and the Haar measure of $G$ is in the same measure
class as the push-forward of the product measure on $\overline{V}_{\theta}\times P_{\theta}$
by $\varepsilon$.
\end{lem}
\begin{proof}
This is Lemma IV.2.2 of \citep{margulis1991discrete}.
\end{proof}

Lemma \ref{lem:relations of V P and G} implies that as $G$-spaces
we have $\nicefrac{G}{P_{\theta}}\cong\overline{V}_{\theta}$, and
that the natural map $\nicefrac{G}{P}\to\nicefrac{G}{P_{\theta}}$
corresponds under this isomorphism to the projection $\overline{V}\to\overline{V}_{\theta}$
with respect to the decomposition $\overline{V}=\overline{V}_{\theta}\rtimes\overline{L}_{\theta}$. 

In particular $\overline{V}\cong\nicefrac{G}{P}$ as Borel $G$-spaces.
For certain elements of $G$ it is easy to describe the $G$-action on $\overline{V}$ explicitly.
Let $u\in\overline{V}\cong\nicefrac{G}{P}$. Observe that if $v\in\overline{V}\subset G$
then $v.u=vu$. Similarly if $s\in S\subset G$ then $s.u=sus^{-1}$.

\subsubsection*{An example of $\GG, \PP, \SS$, etc.} 
\label{sub:an example of linear algebraic groups}

For the convenience of the reader we present a concrete and standard example of an algebraic group $\GG$ satisfying the requirements of Theorem \ref{thm:Intermediate-factor-theorem} as well as its subgroups $\PP, \SS$ etc. introduced above. This example is treated in detail in Section II.3 of \cite{margulis1991discrete}.

Let $\GG$ be the algebraic group $\SL{n}$ for some $n \in \N, n \ge 2$. Here $\SL{n}$ is the subgroup of $\GL{n}$ defined by the vanishing of the polynomial equation $\det - 1 = 0$. Note that $\SL{n}$ is defined over the field $\QQ$.

As a maximal split torus of $\SL{n}$ we may take the diagonal subgroup $\SS$
$$ \SS = \{ s \in \mathrm{diag}(s_1, \ldots, s_n) \: : \: s_1 s_2 \cdots s_n = 1 \} $$
So $\SS$ is isomorphic to the product of $n-1$ copies of the multiplicative group.  The roots $\Phi = \Psi(\SS,\GG)$ are given by
$$ \Phi = \{ \alpha_{i,j} \: : \: 1 \le i,j \le n, \; i \neq j \} $$
where the root $\alpha_{i,j}$ is defined by the formula
$ \alpha_{i,j}(s) = s_i s_j^{-1} $.

There is a standard root ordering corresponding to the choice $\Delta = \{\alpha_{i,i+1}\}_{i=1,\ldots,n-1} $ of simple roots. The corresponding parabolic subgroup $\PP$ is the group of upper-triangular matrices of determinant one. The subgroup $\VV$ is the group of unipotent upper-triangular matrices. 

It remains to discuss the groups associated to choices of subsets $\theta  \subset \Delta$ of simple roots. Let us describe explicitly the case where $\theta = \{\alpha_{1,2}\}$, the general situation being similar. Then $\SS_\theta$ consists of diagonal matrices  subject to the condition that $s_1 = s_2$.  The parabolic subgroup $\PP_\theta$ consists of elements of $\SL{n}$ that vanish below the diagonal, except possibly for the $(2,1)$-th entry. $\VV_\theta$ consists of unipotent upper triangular matrices that additionally vanish in the $(1,2)$-th entry. Then $\VV = \VV_\theta \rtimes \mathbb{L}_\theta $ where $\mathbb{L}_\theta$ consists of unipotent matrices with only the $(1,2)$-th coordinate non-zero outside the diagonal. Finally the groups $\overline{\PP}_\theta, \overline{\VV}_\theta$ and $\overline{\mathbb{L}}_\theta$ are nothing but the transpose of $\PP_\theta, \VV_\theta$ and $\mathbb{L}_\theta$, respectively.

\subsection{Contracting automorphisms}
\label{sub:contracting_automorphisms}

Recall that a \emph{contraction} of a topological group is an automorphism
$\varphi$ such that for every compact set $K\subset G$ and every
neighborhood of the identity $U$ there exists $n\in\N$ such that
$\varphi^{n}\left(K\right)\subset U$. 

In the situation of \ref{sub:subgroup_structure}, given a subset $\theta\subset\Delta$ we can find an element $s\in S_{\theta}$ such that $\mathrm{Inn}\left(s\right)$
is contracting on $V_{\theta}$ and is the identity of
$L_{\theta}$ (see II.3.1 in \citep{margulis1991discrete}).
Note that $\mathrm{Inn}\left(s\right)$ is a non-trivial automorphism
of $V$ provided that $\theta\subsetneq\Delta$.


\begin{defn}
\label{def:psi_theta}
Given a collection $\theta\subset\Delta$ of simple roots denote
\[
\psi_{\theta}\left(E\right)=\overline{V}_{\theta}\left(E\cap\overline{L}_{\theta}\right)
\]
where $E\subset\overline{V}$ is a measurable subset.
\end{defn}

So $\psi_\theta(E)$ is a subset of $\overline{V}$ and is always a union of certain $\overline{V}_\theta$-cosets. For example, $\psi_{\emptyset}\left(E\right)$ is either $\overline{V}$
or $\emptyset$ depending on whether $e\in E$ while $\psi_{\Delta}\left(E\right)=E$.
The following lemma (see IV.2.5 in \citep{margulis1978quotient} or
8.2.8 in \citep{zimmer1984ergodic}) controls the image of measurable
subsets of $\overline{V}$ under iterations of $\mathrm{Inn}\left(s\right)$.
\begin{lem}
\label{lem:contraction} Let $\theta \subset \Delta$ be a collection of simple roots and $s\in S_{\theta}$ an element such that $\mathrm{Inn}(s)$ is contracting on $V_\theta$. If $C\subset\overline{V}$ is a measurable subset then for almost
every $u\in\overline{V}$ we have
\[
\mathrm{Inn}\left(s^{n}\right)\left(uC\right)=s^{n} uC  s^{-n}\longrightarrow\psi_{\theta}\left(uC\right)
\]
where the convergence is understood in measure.
\end{lem}

Convergence in measure is discussed in Subsection \ref{sub:convergence in measure} below. We remark that the proof actually relies on the fact that $\mathrm{Inn}(s)$ is \emph{expanding} (i.e. is an inverse of contracting) on $\overline{V}_\theta$.

\subsection{Mautner's lemma}
\label{sub:Mautner's lemma}

Let $\GG$ be a simply-connected semisimple algebraic $k$-group  as in the statement of Theorem \ref{thm:Intermediate-factor-theorem}. The following is a well-known and immediate corollary
of Mautner's lemma.

\begin{lem}
	\label{lem:cor of Mautner lemma}Let $X$ be an irreducible
	$\GG_k$-space with finite $\GG_k$-invariant measure, and let $s\in S$ be
	such that $\text{Inn}\left(s\right)$ is contracting on some
	$V_{\theta}$ with $\theta \subsetneq \Delta$  Then the action of $s$ on $X$ is ergodic.
\end{lem}

For the proof see for example II.3.3 of \citep{margulis1991discrete}. A recent treatment of Mautner's lemma is given in \cite{bader2014equicontinuous}.

\section{Measure algebras and the Effros Borel structure}
\label{sec:measure algebras}

A central tool in the proofs of the factor and the intermediate factor theorems is that of measure algebras. Indeed, both theorems admit an equivalent formulation in that language. The restatement relies on Mackey's point realization and its extensions to $G$-spaces.

Therefore Section \ref{sec:measure algebras} is dedicated to measure algebras, in particular those that arise in certain product spaces. We consider the Effros Borel structure on the collection of measure sub-algebras of a given algebra. This depends on regarding a measure algebra with the convergence in measure topology and observing that a sub-algebra becomes a closed subspace.

To be concrete, we will be interested in understanding the measure algebras associated to the following situation. Assume that three $G$-spaces $X,Y$ and $Z$ are given as well as a pair of $G$-maps  $f:X\times Z\to Y$ and $g:Y\to X$ whose composition is the projection.
For the purposes of Theorem \ref{thm:Intermediate-factor-theorem} we will only require
the case $Z=\nicefrac{G}{P}$, but presently we discuss the general case. 

One main goal of the current section is to prove Theorem \ref{thm:iota is measurable} below.

\subsection{Measure algebras} \label{sub:measure_algebras}

Recall that the spaces $X,Y$ and $Z$ are assumed to be standard Borel, and let $\eta$ be a probability measure on $X$.

\begin{defn}
\label{def:measure algebra}
The \emph{measure algebra} $\malg{X}$ is the Boolean algebra of all measurable
subsets up to almost everywhere equivalence.
\end{defn}

In other words  $\malg{X}$ is just the $\sigma$-algebra of $X$ modulo the ideal of $\eta$-null sets. The algebra $\malg{X}$ clearly depends on the class of $\eta$.

In the situation under consideration we have a pair of $G$-maps 
$$ X \times Z \xrightarrow{f} Y \xrightarrow{g} X $$ 
Passing to measure algebras we obtain the reverse inclusions
$$\malg{X} \subset \malg{Y}\subset \malg{X\times Z} $$
We would like to decompose the intermediate measure algebra $\malg{Y}$ over $X$ in such a way that the fiber over every $x \in X$ becomes a sub-measure algebra of $\malg{Z}$. 

\subsubsection*{Abelian von-Neumann algebras}
Working towards such a decomposition we discuss an equivalent realization of measure algebras, as following.
Consider the Hilbert space $L^2(X,\eta)$ and denote
$$\vnalg{X} = L^\infty(X,\eta) \subset \mathcal{L}(L^2(X,\eta))$$
where the element of $\vnalg{X}$ represented by the function $f \in L^\infty(X,\eta)$ is regarded as a bounded operator on $L^2(X,\eta)$ corresponding to multiplication by $f$. In fact $\vnalg{X}$ is an \emph{abelian von-Neumann algebra} and $\malg{X}$ is isomorphic as a Boolean algebra to the sub-algebra of (orthogonal) projections in $\vnalg{X}$. See Chapter I.7 of \cite{dixmier} for more information regarding abelian von-Neumann algebras.

\subsubsection*{Direct integral decomposition}

The abelian von-Neumann algebras corresponding to the $G$-spaces $X, Y$ and $Z$ stand in the following relation
$$\vnalg{X} \subset \vnalg{Y}\subset \vnalg{X\times Z}  $$
There is a theory of direct integrals of von-Neumann algebras, as exposed e.g. in part II of \cite{dixmier}. In particular we may write
$$ \vnalg{X\times Z} \cong \vnalg{X} \otimes \vnalg{Z} \cong \int^\oplus \vnalg{Z} \; d\eta(x) $$
Note that under this identification $\vnalg{X}$ corresponds to the diagonalizable operators\footnote{An operator on the Hilbert space $\int^ \oplus \mathcal{H}(x) \, d\mu(x) $ is \emph{diagonalizable} if it corresponds to point-wise multiplication by a function in $L^\infty(X,\mu).$} in the direct integral decomposition on the right hand side. Using the fact that $\vnalg{X} \subset \vnalg{Y}$ and according to Theorem 2 on p. 198 of \cite{dixmier} we obtain a decomposition
$$ \vnalg{Y} = \int^\oplus \vnalg{Y}(x) \; d\eta(x) $$
where for $\eta$-almost every $x \in X$
$$\vnalg{Y}(x) \subset \vnalg{Z}$$
This decomposition is unique up to $\eta$-null sets by Proposition 1, p.197 of \cite{dixmier}.

\subsubsection*{Equivariance} The $G$-actions on the three spaces give rise to $G$-actions on the corresponding abelian von-Neumann algebras $\vnalg{X}, \vnalg{Y}$ and $\vnalg{X \times Z}$. Clearly $\vnalg{X}$ and $\vnalg{Y}$ are $G$-invariant as sub-algebras of $\vnalg{X \times Z}$.
Fix an element $g \in G$. By uniqueness of decomposition it follows that $$g \vnalg{Y}(x) = \vnalg{Y}(gx)$$
holds for $\eta$-almost every $x \in X$. 


\subsubsection*{Projection operators} Consider a projection operator $ P \in \vnalg{Y}$ so that $P^2 = P$ and $P^* = P$. We obtain a direct integral decomposition
$$ P = \int^\oplus P(x) \; d\eta(x) $$
where $\eta$-almost everywhere $P(x)$ is an element of $\vnalg{Y}(x)$. The identities
$$ P^2(x) = P(x), \quad P^{*}(x) = P(x) $$
hold for $\eta$-almost every $x \in X$ (see Proposition 3, p. 182, \cite{dixmier}) and so almost every $P(x)$ is an orthogonal projection in $\vnalg{Z}$. 

Note that the above decomposition is only well-defined up to $\eta$-null sets. This however will suffice for our purposes.

\subsubsection*{Direct integrals of measure algebras}

\begin{defn}
\label{def:the decomposition B(x)}
Let $\malg{Y}(x)$ be the Boolean sub-algebra of $\malg{Z}$ given by
$$ \malg{Y}(x) = \vnalg{Y}(x) \cap \malg{Z} $$
for those $\eta$-almost every $x \in X$ where $\vnalg{Y}(x) \subset \vnalg{Z}$. Here $\malg{Z}$ is identified with the sub-algebra of projections in $\vnalg{Z}$.
\end{defn}

From the above discussion concerning equivariance, we see that the family $\malg{Y}(x)$ is $G$-equivariant in the sense that 
$$ \malg{Y}(gx) = g\malg{Y}(x) $$
holds for every $g \in G$ and $\eta$-almost every $x \in X$. 

\begin{defn}
\label{def:decomposition of $C$}
Given an element $C \in \malg{Y}$ let $P_C \in \vnalg{Y}$ be the corresponding projection. We will write
$$ C = \int^\oplus C(x) \; d\eta(x) $$
where the $C(x) \in \malg{Y}(x)$ are the elements of $\malg{Z}$ corresponding to the projections $P_C(x) \in \vnalg{Y}(x)$ arising in the decomposition of $P_C$ as in Definition \ref{def:the decomposition B(x)}.
\end{defn}

We note once more that the decomposition of an element $C \in \malg{Y}$ as in Definition \ref{def:decomposition of $C$} is  only well-defined up to $\eta$-null sets.

\subsection{The topology of convergence in measure on $\malg{X}$}
\label{sub:convergence in measure}
There is a well-known way to put a metric on the measure algebra $\malg{X}$ associated to the standard Borel space $X$. Namely, fix a probability measure $\mu$ in the measure class on $X$. Given a pair $E,F\in \malg{X}$ we can now set $$d\left(E,F\right)=\mu\left(E\triangle F\right)$$
$d$ is clearly a metric  and it defines the \emph{topology of convergence in measure} on $\malg{X}$.

\begin{rem*}
	The use of this topology was suggested by Nevo and Zimmer in Section
	4 of \citep{nevo2002generalization}.
\end{rem*}

\begin{prop}
	\label{prop:properties of BX}
	Every Boolean sub-algebra $\malg{}\subset \malg{X}$  is complete with respect to $d$ and is therefore closed in the topology of convergence in measure.
	Moreover $\malg{X}$ with the convergence in measure topology is second countable and Polish.
\end{prop}

\begin{proof}
	For the completeness of every sub-algebra $\malg{}$, see Proposition 2.30 of \citep{folland1984real}. In particular $\malg{X}$ is metric and complete with respect to $d$.

	Since $\left(X,\mu\right)$ is a standard Borel probability space, $\mu$ is regular (Theorem 17.10, \cite{kechris1995classical}). To see that $\malg{X}$ is separable consider the countable collection of all finite unions of basic open sets in $X$. Being separable and metric, $\malg{X}$ is second countable. This implies that $X$ is Polish.

\end{proof}

\begin{rem*} \label{rem:BX not locally compact}
	$\malg{X}$ with the above topology is not compact. Namely,
	assume without loss of generality that $X\cong\left[0,1\right]$ and consider the sequence $f_{n}=\chi_{E_{n}}$
	for $n\ge1$ where $E_{n}\subset\left[0,1\right]$ is the subset of
	all real numbers with $1$ at the $n$-th position of their terminating
	binary expansion. Then $f_{n}$ has no converging subsequence. By
	a similar argument, $\malg{X}$ is not locally compact.
\end{rem*}

\subsection{The Effros Borel space}
\label{sub:Effros-Borel-space}

Let $\mathcal{X}$ be an arbitrary topological space and let $\mathcal{C}(\mathcal{X})$ denote the collection of all closed subsets of $\mathcal{X}$. 

\begin{defn}
\label{def:Effros Borel space}
The \emph{Effros Borel space} associated to $\mathcal{X}$ is the space $\mathcal{C}(X)$ endowed with the $\sigma$-algebra generated by the sets 
$$ M_U = \{  F\in \mathcal{C}(X) : F \cap U = \emptyset \} $$
for every open $U$ in $\mathcal{X}$. 
\end{defn}

\begin{prop} \label{prop:Effros is standard}
	If $\mathcal{X}$ is Polish then the Effros Borel space $\mathcal{C}(\mathcal{X})$ is standard. Moreover if $U_n$ is a countable basis for $\mathcal{X}$ then the $M_{U_n}$ generate the Effros Borel $\sigma$-algebra.
\end{prop}
\begin{proof}
	See Theorem 12.6 on page 75 of \cite{kechris1995classical}.
\end{proof}

\begin{defn}
\label{def: beta}
Assume that $\mathcal{X}$ admits a basis $\beta$ with $ \beta = \{U_i\}_{i \in I} $. For a closed subset $F \subset \mathcal{X}$ denote 
$$ \beta(F) = \{U \in \beta \: : \: F \cap U  = \emptyset \} = \{U \in \beta \: : \: F \in M_U  \}$$
\end{defn}

By the above for every closed subset $F \subset \mathcal{X}$ we clearly have that $F^c = \cup _ {U \in \beta(F)} U $.

\begin{lem}
\label{lem:Effros convergence weaker form}
Assume that $\mathcal{X}$ admits a countable basis $\beta$. Let $F_n \in \mathcal{C}(\mathcal{X})$ be a sequence of closed subsets and $x_n \in F_n$ a sequence of points such that $x_n \to x \in \mathcal{X}$. 

Let $F \subset \mathcal{X}$ be a closed subset such that for every $U \in \beta(F)$ we have $F_n \in M_U$ for all $n$ sufficiently large. Then $x \in F$.
\end{lem}

\begin{proof}
Assume towards contradiction that $x \notin F$. In particular $x \in U$ for some $U \in \beta(F)$. By the assumption of the lemma $F_n \in M_U$ for all $n$ sufficiently large, or equivalently $F_n \cap U = \emptyset$. This contradicts the fact that $x_n \to x \in U$.
\end{proof}

This elementary lemma will play a key role in the proof of the "main lemma"  in Section \ref{sec:ift} below.

\begin{rem}
\label{rem:on chaubuty and effros borel}
The Chaubuty topology was mentioned in Subsection \ref{sub:applications of the stuck-zimmer theorem} of the introduction in connection with invariant random subgroups. In fact, the Borel structure of the Chaubuty topology is just the Effros Borel structure restricted to the space of closed subgroups.
\end{rem}

\subsection{The map $\iota$}
 
The measure algebra $\malg{Y}$ admits a direct integral decomposition $ \malg{Y} = \int ^ \oplus_X \malg{Y}(x) \; d\eta(x)$ as in Definition \ref{def:the decomposition B(x)}. Recall that $\malg{Y}(x)$ is a sub-algebra of $\malg{Z}$ for $\eta$-almost every $x \in X$. In addition, given a fixed element $g \in G$ the equivariance property  $\malg{Y}(gx) = g\malg{Y}(x)$ holds for $\eta$-almost every $x$.

We regard the space $\malg{Z}$ with the convergence in measure topology (see Subsection \ref{sub:convergence in measure}) so that every Boolean sub-algebra of $\malg{Z}$ is closed  by Proposition \ref{prop:properties of BX}. Next, consider the associated Effros Borel space $\mathcal{C}(\malg{Z})$, which is in fact standard Borel by Propositions \ref{prop:properties of BX} and \ref{prop:Effros is standard}. 

\begin{defn}
\label{def:the mapping iota}
Let $\iota$ denote the following map
$$\iota:X \to \mathcal{C}\left(\malg{Z}\right), \quad \iota : x \mapsto \malg{Y}(x) $$ 
\end{defn}

A priori $\iota$ is only defined  $\eta$-almost everywhere. However, by assigning $\iota$ an arbitrary constant value (such as $\malg{Z} \in \mathcal{C}(\malg{Z})$) on the $\eta$-null set where $\iota$ is undefined, we may assume it is in fact defined everywhere. This change clearly does not affect the equivariance property.

A well-known technical trick allows us to pass to a $\eta$-conull and $G$-invariant subset $X_0 \subset X$ such that $g\iota(x) = \iota(gx)$ holds for \emph{every $g \in G$} and every $ x\in X_0$ (see e.g. \cite{zimmer1984ergodic}, Appendix B, Proposition 5). We will continue using the notation $X_0$ for this well-behaved conull subset throughout.

The central result of this section is the following

\begin{thm}
\label{thm:iota is measurable}
The mapping 
$\iota:X_0 \to \mathcal{C}(\malg{Z}) $ is Borel measurable.
\end{thm}

\subsubsection*{Disintegration of measures} Recall that we are considering a sequence of $G$-maps 
$$ \left(X\times Z, \, \eta\times\mu\right) \xrightarrow{f} \left(Y, \, \lambda \right) \xrightarrow{g} (X,  \, \eta ) $$
whose composition is the projection and where we have indicated the respective measures. 
The proof of Theorem \ref{thm:iota is measurable} will depend on a certain 
disintegration\footnote{A detailed account on disintegration of measure can be found in Section 452 of \cite{fremlin2000measure}.} of the involved measures $\eta, \mu$ and $\lambda$. To simplify notation set $m=\eta\times\mu$. 

On the one hand, it is possible to disintegrate the measure $m$ over the projection map $g \circ f$ to $(X,\eta)$ and obtain
$$ m = \int_X \mu_x \; d\eta(x) $$
By uniqueness of disintegration and Fubini's theorem it follows that $\mu_x = \delta_x \times \mu$ holds for $\eta$-almost every $x\in X$. Here  $\delta_x$ denotes the Dirac measure centered at $x$.

On the other hand, we may disintegrate separately over $f$ and $g$. This gives
$$ m =\int_Y m_y \; d\lambda(y) \quad \text{and} \quad \lambda = \int_X \lambda_x \; d\eta(x)  $$
These different disintegrations are related by the following proposition.

\begin{prop} The formula
\label{prop:disintegration formula}
$$ \mu_x = \int_Y m_y \; d\lambda_x $$
holds for $\eta$-almost every point $x\in X$.
\end{prop}

We would like to thank the anonymous referee for suggesting a simplified proof\footnote{It is clear that an analogue of Proposition \ref{prop:disintegration formula} holds more generally given any sequence of $G$-maps of the form $S_1 \xrightarrow{f_1} S_2 \xrightarrow{f_2} S_3$.}.

\begin{proof}
Combining the two disintegrations over $f$ and over $g$ we obtain
$$
m  = \int_{Y} m_{y} \; d \lambda\left(y\right) = \int_{X} \left(\int_{Y} m_{y} \; d\lambda_{x}\left(y\right)\right) \; d\eta\left(x\right)
$$
This implies that the assignment
$$ x \mapsto \int_Y m_y \; d\lambda_x(y) $$
is a disintegration of $m$ over the projection to $(X,\eta)$. 
Uniqueness of disintegration gives the desired conclusion.

\end{proof}

The probability measure $m_y$ is supported on the fiber 
$$ f^{-1}(y) = \{g(y)\} \times Z $$
for $\lambda$-almost every $y \in Y$. Identifying this fiber with $Z$ we may regard the $m_y$ as probability measures on $Z$. In the same way, each probability measure $\mu_x$ is supported on $\{x\} \times Z$ and so can be identified with $\mu$. We make these identifications implicitly below. Having established the required machinery, we turn to

\begin{proof}[Proof of Theorem \ref{thm:iota is measurable}]
The measure algebra $\malg{Z}$ admits a countable basis $\beta$ by Proposition \ref{prop:properties of BX}. As the sets $\{M_U\}_{U\in\beta} $ generate the Effros Borel structure according to Proposition \ref{prop:Effros is standard} it suffices to verify that $\iota^{-1}M_{U}$ is measurable for every $U \in \beta$.

Let $d$ be a metric generating the convergence in measure topology on $\malg{Z}$. Given an element $A \in \malg{Z}$ we let $N_\epsilon(A)$ denote the $\epsilon$-ball at $A$ with respect to $d$. 

Now consider a fixed basic open set $ U \in \beta$. We may assume that $U=N_{\varepsilon}\left(A\right)$
for some $A\in \malg{Z}$ and $\varepsilon>0$. Note that 
\[
x \in (\iota^{-1}M_{U})^c\,\Leftrightarrow\, \malg{Y}(x) \in (M_{U})^c \,\Leftrightarrow\,  \malg{Y}(x) \cap U\neq \emptyset \,\Leftrightarrow\, d\left(\malg{Y}(x),A\right)<\varepsilon
\]
The proof amounts to showing that the condition $d\left(\malg{Y}(x),A\right)<\varepsilon$
depends measurably on $x$ as a function on $X_0$. 

Let $\Phi_A \subset Y$ be the subset defined by
$$\Phi_{A} = \left\{ y\in Y\,:\, m_{y}\left(A\right)\ge\ \nicefrac{1}{2} \right\} $$
so that  $\Phi_{A}$ consists of those  $y\in Y$ such that
$A$ meets the fiber $f^{-1}\left(y\right)$ at more than half of
its $m_{y}$-measure. Since $m_{y}$ is a disintegration of $m$ and passing to measure algebras we may regard $\Phi_A$ as an element of $\malg{Y}$. As such $\Phi_A$ admits a direct integral decomposition (as in Definition \ref{def:decomposition of $C$})
$$ \Phi_A = \int^\oplus \Phi_A(x) \; d\eta(x) $$ with $\Phi_A(x) \in \malg{Y}(x)$
 for almost every $x \in X$.
Let $X_{A,\epsilon} \subset X_0$ be the subset
$$
X_{A,\epsilon} =\left\{ x \in X_0 \,:\, d\left(\Phi_A(x) , A \right) < \varepsilon \quad \text{and} \quad \Phi_A(x) \in \malg{Y}(x) \right\} 
$$
By Fubini's theorem and passing to measure algebras the subset $X_{A,\epsilon}$ belongs to $\malg{X}$. We  claim that 
$$ \left( \iota^{-1} M_{U} \right) ^ c = X_{A,\epsilon}$$  
Note that establishing this claim completes the proof.

In one direction, every $x\in X_{A,\epsilon}$ satisfies  $d\left(\Phi_A(x), A\right)<\varepsilon$ by definition. As $\Phi_A(x) \in \malg{Y}(x)$ this implies that  $d\left(\malg{Y}(x), A\right)<\varepsilon$ as well.

Conversely, consider a point $ x\in (\iota^{-1}M_{U})^c$. So $d\left(\malg{Y}(x), A\right)<\varepsilon$ and there exists some element $C \in \malg{Y}(x)$
with $d\left(C,A\right)<\varepsilon$. We need to show that $d(\Phi_A(x), A) < \epsilon$ holds as well.

Recall that $m_{y}$ is supported on the fiber $f^{-1}(y)$ and can be regarded as a probability measure on $Z$ for $\lambda$-almost every $y \in Y$. Since the elements $C$ and $\Phi_A(x)$ of $\malg{Z}$ belong to the measure sub-algebra $\malg{Y}(x)$ we deduce that
 $$m_y(C) \in \{0,1\} \quad \text{and} \quad m_y(\Phi_A(x)) \in \{0,1\}$$ for $\lambda_x$-almost every $y \in Y$.

Note that $m_y(\Phi_A(x)) = 0 $ if and only if
$m_{y}\left(A\right)\le\frac{1}{2}$. This gives the estimate 
$$ m_{y}\left(\Phi_A(x) \triangle A\right) = \min \{m_y(A), 1-m_y(A)\} \le m_{y}\left(C \triangle A\right)$$
for $\lambda$-almost every $y\in Y$. Integrating over $\lambda_x$ and applying Proposition \ref{prop:disintegration formula} and the remarks following it we obtain
$$
d\left(\Phi_A(x),A\right) = \mu_x \left(\Phi_A(x) \triangle A\right) \le \mu_x \left(C \triangle A\right) = 
d\left(C,A\right)<\varepsilon
$$
as required.
\end{proof}

\section{The Nevo-Zimmer intermediate factor theorem\label{sec:proof of IFT}}
\label{sec:ift}

In the current section we provide a proof of our main result Theorem \ref{thm:Intermediate-factor-theorem}. 

As mentioned in the introduction, we aim to stay as close as possible to Zimmer's proof for the real case \cite{zimmer1982ergodic}, which it turn follows along the lines of the proof of the factor theorem by Margulis \cite{margulis1978quotient}. For this reason what follows below could perhaps be better appreciated in light of and in comparison with this proof by Margulis (see \cite{margulis1978quotient} or Chapter 8 of \cite{zimmer1984ergodic}). 

Our contribution to the proof of the intermediate factor theorem consists in giving a new proof of a certain "main lemma" (Lemma \ref{lem:main lemma}) which is discussed in Subsection \ref{sub:the main lemma} and proved in Subsection \ref{sub:proof_the main lemma} below. Building on this lemma the intermediate factor theorem follows exactly as in Zimmer's proof. Indeed, Subsection \ref{sub:proof of ift}  is dedicated to explaining how Theorem \ref{thm:Intermediate-factor-theorem} is deduced from Lemma \ref{lem:main lemma}.

We retain the notations and assumptions of Theorem \ref{thm:Intermediate-factor-theorem}
throughout the current section.

\subsection{The main lemma}
\label{sub:the main lemma}

The following lemma is nothing but the natural generalization of Margulis' Lemma 1.14.1 of \cite{margulis1978quotient} (reproduced as Lemma 8.3.2 in the book \cite{zimmer1984ergodic}) to the context of the intermediate factor theorem. See also Lemma 4.3 of \cite{zimmer1982ergodic}.

The function $\psi_\theta$ that appears in the statement of the lemma is introduced in Definition \ref{def:psi_theta}. The measure algebra $\malg{Y}$ and its direct integral decomposition are discussed in Subsection \ref{sub:measure_algebras} above. In particular $\malg{Y}(x)$ is almost always a sub-algebra of $\malg{G/P}$. Finally, recall from Subsection \ref{sub:subgroup_structure} that $\malg{G/P}$ can be identified with $\malg{\overline{V}}$ and we make this identification implicitly below.

\begin{lem}[The main lemma]
	\label{lem:main lemma}Assume that a collection $\theta\subsetneq\Delta$ of simple roots is given and let $P\subset P_{\theta}\subsetneq G$ 
	be the corresponding parabolic subgroup. Let $C\in \malg{Y}$ and consider a direct integral decomposition
	$C=\int^{\oplus}C(x) \, d\eta(x)$.  	
	Then for almost every $\left(x,u\right)\in X\times\overline{V}$
	we have  
	$$g \psi_{\theta}\left(u C\left(x\right) \right)\in \malg{Y}(x)$$
	for all $g\in G$.
\end{lem}

To establish Lemma \ref{lem:main lemma} we are required to show that a certain element $ g \psi_\theta(uC(x))$ of $\malg{\nicefrac{G}{P}}$ belongs to the  sub-algebra $\malg{Y}(x)$.
Roughly speaking, the strategy will be to consider $\malg{\nicefrac{G}{P}}	$
with the convergence in measure topology and let $\mathcal{C}(\malg{\nicefrac{G}{P}})$ be the associated Effros Borel space. Using Lemma \ref{lem:contraction} we exhibit this required element as a limit of certain elements $x_n$ belonging to the sub-algebras $\malg{Y}(h_n x)$ of $\malg{\nicefrac{G}{P}}$ for some sequence $h_n \in G$. In this situation Lemma \ref{lem:Effros convergence weaker form} can be used to deduce that in fact the limit belongs to $\malg{Y}(x)$. 

To make this strategy precise we  require several definitions and propositions preceding the proof of Lemma \ref{lem:main lemma}.

\subsubsection*{The sets $N_\alpha$}

Let $\beta$ be a countable basis for the convergence in measure topology on $\malg{\nicefrac{G}{P}}$ endowed with some fixed well-order and recall the $G$-equivariant Borel map $\iota $ constructed in Definition \ref{def:the mapping iota}.

\begin{defn}
	\label{def:N_alpha}
	For a family  $\alpha \subset \beta$ of basic open sets 
	$$
	N_\alpha = \iota^{-1} \bigcap_{U \in \alpha} M_U 
	= \left \{ x \in X \,:\, \malg{Y}(x) \cap U =\emptyset, \;  \forall U \in \alpha \right\} \subset X
	$$
	where $M_U \subset\mathcal{C}(\malg{\nicefrac{G}{P}})$ is the corresponding
	Effros Borel set as in Definition \ref{def:Effros Borel space}.
\end{defn}

According to Theorem \ref{thm:iota is measurable} the subsets $N_\alpha$ for $\alpha \subset \beta$ are measurable.

\begin{defn}
	Given a point $x \in X$ let $\alpha_n(x) $ denote the first $n$ elements of $\beta(\iota x) $  with respect to the fixed well-order on the basis $\beta$ (see Definition \ref{def: beta} for the notation $\beta(\cdot)$).
\end{defn}

Clearly $N_{\alpha_n(x)}$ for $n \in \N$ is a decreasing sequence of measurable subsets of $X$ containing the point of $x$. Informally these should be thought of as a "neighborhood basis" at $x$.

\begin{prop}

	\label{prop:X_2 is conull}
	Let $\eta$ denote the invariant probability measure on $X$. The subset 
	$$ X_1 = \{x \in X \: : \: \text{$\eta(N_{\alpha_n(x)}) > 0$ for all $n \in \N$} \} $$
	is $\eta$-conull in $X$.
\end{prop}
\begin{proof}
	Let $\mathcal{P}_0$ be the subset of the power-set $\mathcal{P}(\beta)$  consisting of all the \emph{finite} families $\alpha \subset \beta$ of basic open set such that $\eta(N_\alpha) = 0$. As $x \in N_{\alpha_n(x)}$ for every $n \in \N$ we clearly have that
	$$ X \setminus X_1  \subset \cup_{\alpha \in \mathcal{P}_0} N_\alpha $$
	Since $\mathcal{P}_0$ is countable the required conclusion follows.
\end{proof}

Informally, $X_1$ should be thought of as the "support" of $\eta$ with respect to the "neighborhoods" $N_{\alpha_n(x)}$. 

\subsection{An ergodicity argument}
\label{sub:ergodicity argument}

Fix some collection $\theta\subsetneq\Delta$ of simple roots and let $s\in S_{\theta}$ be an element such that $\textrm{Inn}\left(s\right)$ is contracting
on $V_{\theta}$. Such an element is guaranteed to exist, as remarked in Subsection \ref{sub:contracting_automorphisms}. Let $\left<s\right> \le G$ denote the cyclic subgroup generated by $s$.

 A key ingredient in the proof of Lemma \ref{lem:main lemma} is the study of the space $X\times G$ with the $G \times \left<s\right>$-action given by 
$$\left(h,n\right).\left(x,g\right)=\left(hx,hgs^{-n}\right)$$

\begin{prop}
\label{pro:GxZ acts ergodically}
The action of $G\times \left<s\right>$ on $X\times G$ is ergodic.
\end{prop}
(Compare Proposition 2.2.2 of \cite{zimmer1984ergodic}).
\begin{proof}
Since the $G$-action on $X$ is irreducible and finite measure
preserving, it follows from Mautner's lemma (reproduced here as Lemma \ref{lem:cor of Mautner lemma}) that the $s$-action on $X$ is ergodic as well. 

Let $E\subset X\times G$
be a measurable subset invariant under the given action of $G\times\left<s\right>$.
In particular $E=\pi^{-1}\left(\overline{E}\right)$ where $\pi$
is the natural projection of $X\times G$ on $X\times\nicefrac{G}{\left<s\right>}$
and $\overline{E}$ is $G$-invariant. For $\overline{g} \in\nicefrac{G}{\left<s\right>}$
denote 
$$E_{\overline{g}}=\left\{ x\in X\,:\,\left(x,\overline{g} \right)\in \overline{E}\right\} $$
$G$-invariance of $\overline{E}$ implies that $h E_{\overline{g}}=E_{h\overline{g}}$ for every $h \in G$
which in turn gives $s E_{\overline{e}}={E}_{\overline{e}}$. The ergodicity of $s$ implies that $E_{\overline{e}}$ is either null or co-null in $X$, and the proposition follows.
\end{proof}

To facilitate the application of Lemma \ref{lem:Effros convergence weaker form} we introduce the following notation.

\begin{defn}
\label{def:the set ZU}
Given a measurable subset $N\subset X$ let $Z_{N}\subset X\times G$ be given
by
$$
Z_{N} = \left \{ \left(x,g\right)\in X\times G \: : \:  \overline{\left\{ hg s^{-n}\right\} } = G, \; \text{where $n\ge 0, h\in G$ and $ hx \in N$} \right\} 
$$

\end{defn}

In other words $Z_N$ consists of all pairs $(x,g) \in X \times G$ such that 
the set $\{hg s^{-n}\}$ is dense in $G$ where $h \in G$ ranges over these elements sending $x$ inside $N$. 

\begin{prop}
\label{prop:Z_F is conull}
$Z_{N}$ is conull in $X\times G$ for every measurable $N\subset X$ with $\eta(N) > 0$.
\end{prop}

(Here $\eta$ denotes as usual the invariant probability measure on $X$).

\begin{proof}
Let $\mu$ be a Haar measure
on $G$ and $\mathcal{V} $ a countable basis
for $G$. 
Every $V \in \mathcal{V}$ is open and so clearly $\left(\eta\times\mu\right)\left(N\times V\right)>0$. 
Observe that with respect to the given $G\times\left<s\right>$-action on $X\times G$ the subset
$$
W_{V}=\left(G\times s^{\N} \right).\left(N\times V\right)
$$
satisfies 
$$\left(\mathrm{id}_{G}, s\right).W_{V}\subset W_{V}$$

Since the action is probability measure preserving this implies that
 $W_{V}$ is $G\times\left<s\right>$-invariant, up to a null set. By the ergodicity established in Proposition \ref{pro:GxZ acts ergodically} the set $W_V$ is conull. So the countable  intersection $W=\cap_{V \in \mathcal{V} }W_{V}$ is conull as well. 
 
Now almost every pair $\left(x,g\right)\in X\times G$ meets the $G\times s^{\N} $-orbit of $N\times V$ for every $V \in \mathcal{V}$. In other words for every $V \in \mathcal{V}$ and for almost every such pair there are $h\in G$, $n\in\N$ with $hx\in N$ and $hgs^{-n}\in V$, as required.
\end{proof}

\subsection{Proof of the main lemma}
\label{sub:proof_the main lemma}

\begin{proof}[Proof of Lemma \ref{lem:main lemma}]

Consider the sets $Z_{N_\alpha}$ associated to finite families $\alpha \subset \beta$ of basic open sets (see Definition \ref{def:the set ZU}). By Proposition \ref{prop:Z_F is conull} these subsets $Z_{N_\alpha}$ of $X \times G$ are $(\eta\times\mu)$-conull as long as $\eta(N_\alpha) > 0$. Let $Z$ be the countable intersection over all the \emph{conull} $Z_{N_\alpha} $ of this form so that $Z$ is conull in $X \times G$ as well.


Now let $Z_{1}=Z\cap (X\times\overline{V})$. In fact, $Z_{1}$ is conull
in $X\times\overline{V}$ with the measure $\eta\times\mu_{\overline{V}}$
where $\mu_{\overline{V}}$ is the Haar measure on $\overline{V}$.
This is proved using the argument of Lemma 1.9 in \citep{margulis1978quotient}.

Recall that $\nicefrac{G}{P}$ can be identified with $\overline{V}$ as a $G$-space up to null sets.  Similarly, 
we identify the measure algebra $\malg{\nicefrac{G}{P}}$ with $\malg{\overline{V}}$.
Let 
\[
Z_{2}=\left\{ \left(x,u\right)\in X \times\overline{V}\,:\,\textrm{\ensuremath{s^{n}uC(x)s^{-n}} converges in measure \ensuremath{\eta_{\overline{V}}} to \ensuremath{\psi_{\theta}\left(uC(x)\right)}}\right\} 
\]
which is conull by Lemma \ref{lem:contraction} and Fubini's theorem. 
We now claim that the assertion of the lemma holds for $(x,u)$ in the following conull subset $\mathcal{Z}$ of $X \times \overline{V}$
$$ \mathcal{Z} = \{(x,u^{-1}) \, : \, (x,u)\in Z_1 \} \cap Z_{2} \cap \left(\left(X_0 \cap X_1\right) \times \overline{V}\right) $$

 Fix $g\in G$ and consider some pair $\left(x,u\right)\in \mathcal{Z}$.
As  $x \in X_1$ it follows that $\eta(N_{\alpha_n(x)}) > 0$ for every $n \in \N$ (see Proposition \ref{prop:X_2 is conull}). In particular $Z_{N_{\alpha_n(x)}}$ is conull for every $n \in \N$ and since $\left(x,u^{-1}\right)\in Z_{1}$ we can choose sequences $h_{n}\in G$ and $m_n \in \N$ such that
$$h_{n}u^{-1}s^{-m_{n}}\to g, \quad  h_{n}x\in N_{\alpha_n(x)}, \quad \text{and} \quad m_{n}\to\infty$$
Denote $g_{n}=h_{n}u^{-1}s^{-m_{n}}$. Then
 $$h_{n}=g_{n}s^{m_{n}}u  \quad \text{and} \quad  g_{n}\to g$$ 
Using the fact that $\left(x,u\right)\in Z_{2}$ we see that the sequence $h_n C(x)$ has a limit
\[
h_{n}C(x)=g_{n}s^{m_{n}}uC(x)=g_{n}\left(s^{m_{n}}uC(x)s^{-m_{n}}\right) \xrightarrow{n\to\infty} g\psi_{\theta}\left(u C(x)\right)
\]
in the topology of convergence in measure on $\malg{\overline{V}}$. To conclude that this limit belongs to $\malg{Y}(x)$ it remains to verify that the conditions of Lemma \ref{lem:Effros convergence weaker form} are satisfied with respect to the two sequences
$$x_n = h_n C(x), \quad F_n =  h_n \malg{Y}( x) = h_n \iota(x), \quad x_n \in F_n $$
as well as with respect to $x = g\psi_{\theta}\left(uC(x)\right)$ and $F = \malg{Y}(x)$, using the notations of that lemma. The fact that $x_n$ converges to $x$ has already been established.

As the restriction of $\iota$ to $X_0$ is $G$-equivariant and $x \in X_0$ we have that $F_n = \malg{Y}(h_n x) = \iota(h_n x)$. Moreover, it is immediate from the definitions that for every basic open set $U \in \beta(\iota x) = \beta(F)$ the inclusion $\iota N_{\alpha_n(x)} \subset M_U$ holds for all $n$ sufficiently large.
Finally, recall that the elements $h_n$ were chosen so that $$F_n = \iota(h_n x) \in \iota N_{\alpha_n(x)}$$
This completes the verification needed to apply Lemma \ref{lem:Effros convergence weaker form} and deduce that $x \in F$ as required.

\end{proof}

\subsection{Proof of the Nevo-Zimmer intermediate factor theorem}
\label{sub:proof of ift}

We now provide a proof of Theorem \ref{thm:Intermediate-factor-theorem} relying on the Main Lemma \ref{lem:main lemma}. As mentioned in the introductory remarks to Section \ref{sec:ift}, this is essentially the proof of \citep{zimmer1982ergodic} and we reproduce it
here for the sake of completeness.

Recall that a $G$-equivariant Boolean isomorphism
of measure algebras lifts to a $G$-equivariant measure space isomorphism
between conull subsets of the original spaces (see Appendix B.7 in \citep{zimmer1984ergodic}).
So it suffices to show that every $G$-invariant
measure sub-algebra $\malg{}$ with $\malg{X} \subset \malg{} \subset \malg{X\times\nicefrac{G}{P}}$
equals $\malg{X\times\nicefrac{G}{P_{\theta}}}$ for some collection $\theta\subset\Delta$ of simple roots.

\begin{proof}[Proof of Theorem \ref{thm:Intermediate-factor-theorem}] 
	
Consider a $G$-invariant  measure sub-algebra $\malg{}$ lying in between $\malg{X}$ and $\malg{X\times\nicefrac{G}{P}}$
as above. Let $\malg{0}$ be a maximal measure algebra such
that $\malg{X}\subset \malg{0}\subset \malg{}$ and $\malg{0}=\malg{X \times \nicefrac{G}{P_{\theta}}}$
for some parabolic subgroup $P_\theta$ with $\theta\subset\Delta$. We want to show that $\malg{0}=\malg{}$,
so assume to the contrary that there exists an element $C\in \malg{}\backslash \malg{0}$.
It is certainly possible that $P_{\theta}=G$ but as $\malg{0}$ is a
proper subalgebra we must have $P\subsetneq P_{\theta}$.

Using the direct integral decomposition discussed in Subsection \ref{sub:measure_algebras} we write
$$ \malg{} =  \int^{\oplus} \malg{}(x) \; d\eta(x) \quad \text{and} \quad C = \int^{\oplus} C(x) \; d\eta(x) $$
where $C(x)$ satisfies $C(x)\in \malg{}(x)$ for every $\eta$-almost every $x\in X$. Since $C\notin \malg{0}$ but $C\in \malg{X \times \nicefrac{G}{P}}$
we may assume that $C(x)\notin \malg{\nicefrac{G}{P_{\theta}}}$
for $x$ in a set of positive measure (see e.g. Lemma 5.5 of \citep{varadarajan2007geometry}).

Recall that $P_{\theta}$ is generated by the parabolic subgroups
$P_{\left\{ b\right\} }$ for $b\in\theta$, i.e. the minimal parabolic
subgroups $P'$ contained in $P_{\theta}$ such that $P\lneq P'$
(see I.1.2.4 of \citep{margulis1991discrete}). Observe that for a
measurable subset $E\in \malg{\nicefrac{G}{P}}$ we have $E\in \malg{\nicefrac{G}{P_{\theta}}}$
if and only if the preimage $\overline{E}$ of $E$ in $\malg{G}$
satisfies $\overline{E} P_{\theta}=\overline{E}$.
The two previous facts imply that there is a fixed simple root $b\in\Delta$ such
that $C(x)\notin \malg{\nicefrac{G}{P}_{b}}$ for $x$ in
a set of positive measure.

We emphasize that the $\rank{\GG}k\ge2$ assumption is used at this
point to imply that $P_{b}$ is a proper subgroup of $G$. This is
needed to apply Lemma \ref{lem:main lemma}, or more concretely to establish the existence of some $s\in S_{\theta}$ such that $\mathrm{Inn\left(s\right)}$
is contracting on $V_{\theta}$.

Recall the identifications of $\malg{\nicefrac{G}{P_{b}}} \subset \malg{\nicefrac{G}{P}}$
with $\malg{\overline{V}_{b}} \subset \malg{\overline{V}}$.
Moreover $\overline{V}=\overline{L}_{b}\ltimes\overline{V}_{b}$ and
the second inclusion corresponds to the pullback of the projection
$\overline{V}\to\overline{V}_{b}$. So we have for $x$ in a set of
positive measure that $uC(x)\cap\overline{L}_{b}$ is neither null
or conull for $u\in\overline{V}$ in a positive measure subset of
$\overline{V}$. We may now apply Lemma \ref{lem:main lemma} and
Fubini's theorem to obtain some particular $u\in\overline{V}$ such
that
\begin{enumerate}
\item $uC(x)\cap\overline{L}_{b}$ is neither null or conull in $\overline{L}_{\theta}$
for $x$ in a set of positive measure.
\item for almost all $x$, $g\psi_{b}\left(uC(x)\right)\in \malg{}(x)$ for
all $g\in G$.
\end{enumerate}
Note that condition (1) implies that for $x$ in a set of positive
measure, $\psi_{b}\left(uC(x)\right) \notin \malg{\nicefrac{G}{P_{b}}}$
and in particular $\psi_{b}\left(uC(x)\right)\notin \malg{\nicefrac{G}{P_{\theta}}}$.

Let $\malg{1}(x)$
be the algebra generated by $\malg{\nicefrac{G}{P_{\theta}}}$
and $g\psi_{b}\left(uC(x)\right)$ for every $g\in G$. Note
that on the one hand from (2) above $\malg{\nicefrac{G}{P_{\theta}}}\subset \malg{1}(x)\subset \malg{}(x)$
for almost every $x$, while on the other hand from (1) above $\malg{1}(x)$
strictly contains $\malg{\nicefrac{G}{P_{\theta}}}$ for $x$
in a subset of positive measure. Moreover $\malg{1}(x)$ is
clearly a $G$-invariant subalgebra of $\malg{\nicefrac{G}{P}}$
whenever it is defined. But every such $G$-invariant subalgebra
must be of the form $\malg{\nicefrac{G}{P_{\theta'}}}$ for
some collection of simple roots $\theta'(x)\subset\Delta$ (see Proposition \ref{prop:invariant under G+}).
As $\Delta$ is finite, there is some fixed collection of simple roots $\theta'\subset\Delta$ with $\theta'\subsetneq\theta$ such that the subset $X' \subset X$
where $\malg{1}(x)=\malg{\nicefrac{G}{P_{\theta'}}}$
 has positive measure.

We have established that $\malg{\nicefrac{G}{P_{\theta'}}}\subset \malg{}(x)$
for every $x\in X'$. Recall that $g\malg{}(x)=\malg{}(gx)$ holds for all $ x \in X_0$. In particular the subset $X'_0$ of $X_0$ consisting of those $x \in X_0$ such that $\malg{\nicefrac{G}{P_{\theta'}}}\subset \malg{}(x)$ is $G$-invariant and has positive measure. By ergodicity $X'_0$ must be conull and
so $\malg{\nicefrac{G}{P_{\theta'}}}\subset \malg{}(x)$  for
almost every $x$. But this contradicts the maximality of $\malg{0}$.
\end{proof}

\section{The Stuck-Zimmer theorem\label{sec:The-Stuck-Zimmer-theorem}}

We sketch an outline of the proof of the Stuck-Zimmer theorem from the intermediate factor theorem. As explained in the introduction, this deduction is given in the work of Stuck and Zimmer \cite{stuck1994stabilizers} and is sufficiently general to apply to the case of local fields. It is repeated here for the reader's convenience.

We work over an arbitrary local field $k$, while keeping in mind that the
original theorem is essentially\footnote{The real Lie group case is discussed in detail in Subsection \ref{sub:the real case}.} the case $k=\R$. For the complete details the
reader is referred to \citep{stuck1994stabilizers}.

\subsection*{Proof of the Stuck-Zimmer theorem}

Let $\GG$ and $G = \GG_k$ be as in the statement of Theorem \ref{thm:stuck-zimmer-over-local-fields}.
We need to show that every faithful properly ergodic and irreducible
$G$-space $X$ with finite invariant measure $\eta$ is essentially free. 
The starting
point is the following lemma. 

\begin{lem}
\label{lem:lemma on weak amenable actions}
Assume that $G$ has (T).  Then the action of $G$
on $X$ is weakly amenable if and only if the action is essentially
transitive.
\end{lem}
\begin{proof}
See Lemma 1.5 of \cite{stuck1994stabilizers}, which is proved for any second countable locally compact group with property (T).
\end{proof}

For a discussion of amenable and weakly amenable actions the reader
is referred to \citep{stuck1994stabilizers,zimmer1984ergodic}. Since
$G$ has property $\left(T\right)$, we may assume that the action
is not weakly amenable.

Next, we let $P=\PP_{k}$ where $\PP$ is a minimal parabolic $k$-subgroup
of $\GG$. We equip $\nicefrac{G}{P}$ with a quasi-invariant measure
$\mu$. $P$ is solvable and hence amenable, and so the $G$ action
on $\nicefrac{G}{P}$ is amenable (see Section 4.3 in \citep{zimmer1984ergodic}).
From this fact combined with the contra-positive of the weak amenability
of the action one obtains a measure space $\left(Y,\lambda\right)$ and
a pair of $G$-equivariant maps
\[
\nicefrac{G}{P}\times X\overset{f}{\to}Y\overset{p}{\to}X
\]
such that $p\circ f$ is the projection on $X$ (for a general formulation of this argument, see Proposition 2.4.5 of \cite{creutz2013stabilizers}). Furthermore $f$
and $p$ are $G$-maps with respect to the measures $\mu\times\eta$,
$\lambda$ and $\eta$ on $\nicefrac{G}{P}\times X$, $Y$ and $X$ respectively. 

We are now in a situation to apply the Nevo-Zimmer intermediate factor theorem
--- namely Theorem 1.1 of \citep{nevo2002generalization} in the setting of real Lie groups, or our Theorem \ref{thm:Intermediate-factor-theorem} for arbitrary $k$. As a result
we may assume that $Y\cong\nicefrac{G}{Q}\times X$ where $Q=\QQ_{k}$
for a parabolic $k$-subgroup $\QQ$ containing  $\PP$, that $f$ and
$p$ are the natural maps and that the $G$-action on $Y$ is the
product action. Furthermore $Q$ is a proper subgroup of $G$ --- this
follows from the lack of weak amenability.

To complete the proof we consider the stabilizers $G_{x}$ for $x\in X$.
Again, from the definition of a weak amenable action we obtain that
for almost every $x\in X$ the stabilizer subgroup $G_{x}$ is acting trivially on $p^{-1}\left(x\right)$.
Hence
\[
G_{x}\subset H_{x}=\bigcap_{gQ}gQg^{-1}
\]
where the intersection is taken over almost every $gQ\in\nicefrac{G}{Q}$. By standard
arguments $H_{x}$ is a normal subgroup of $G$. $H_x$   is  proper
since $Q\lneq G$. In other words,  $\mu$-almost every stabilizer $G_{x}$ is contained
in a proper normal subgroup. To conclude that $G$ is acting essentially freely thereby  completing the proof we invoke 

\begin{lem}
\label{lem:on stabilisers in proper normal subgroups}
Assume that $G$ is acting faithfully and irreducibly. Let $N \lhd G$ be a normal subgroup. Then the centraliser $Z_G(N)$ is acting essentially freely.
\end{lem}
\begin{proof}
This is Lemma 1.8 of \cite{stuck1994stabilizers}, proved there for a general second countable locally compact group $G$.
\end{proof}


\section{Invariant random subgroup rigidity}
\label{sec:invariant random subgroup rigidity}

In virtue of the fact that an invariant random subgroup is nothing but a $G$-space of a particular kind, it is clear that the Stuck-Zimmer theorem provides information on invariant random subgroups. However, to go from Theorem \ref{thm:stuck-zimmer-over-local-fields} to the rigidity result of Corollary \ref{cor:On IRS} several additional arguments are required.

First, we need to show that an homogeneous space which admits an invariant probability measure necessary corresponds to a lattice. Secondly, note that the natural action of $G$ on its space of closed subgroups by conjugation has as the stabilizer of the point $H \le G$ not $H$ itself but rather its normalizer $N_G(H) \le G$.

The following is essentially the proof given in Section 4 of \cite{abert2012growth} (see also \cite{gelander2015lecture}). We repeat it here with minor modifications to treat the case of arbitrary local fields.

\begin{proof}[Proof of Corollary \ref{cor:On IRS}]
Let $k$ be a local field and $\GG$  a connected simply-connected almost $k$-simple algebraic $k$-group. Assume that $\GG$ satisfies the conditions of Theorem \ref{thm:stuck-zimmer-over-local-fields}, so that in particular $\rank{\GG}{k} \ge 2$ and $G = \GG_k$ has property (T).

Let $\mu$ be an irreducible invariant random subgroup of $G = \GG_k$. Consider the Chaubuty space of $G$ endowed with the invariant probability measure $\mu$ and regarded as a $G$-space. As noted above, the stabilizer of every point $H \le G$ in this $G$-space is the corresponding normalizer $N_G(H)$.

It follows from the irreducibility assumption that the action has central kernel.
Therefore Theorem \ref{thm:stuck-zimmer-over-local-fields} applies and we deduce that either $\mu$-almost every stabilizer is central or the action is essentially transitive. The first case is clearly not possible.

It remains to deal with the essentially transitive case. Say that $\mu$-almost every subgroup of $G$ is conjugate to a certain $H \le G$. Moreover there is a $G$-invariant probability measure on $G / \Lambda$ where $\Lambda = N_G(H)$. According to Theorem 2 of \cite{margulis1977cobounded} there is a non-discrete normal subgroup $N \lhd G$ such that $N \le \Lambda$ and the projection of $\Lambda$ to $G/N$ is a lattice. However such a subgroup $N$ would belong to the kernel of the action and hence $N$ must be trivial.  In particular $\Lambda$ is a lattice in $G$ admitting $\Gamma$ as a normal subgroup. Recall that by the normal subgroup theorem $\Gamma$ is  either  of finite index $\Lambda$ and hence a lattice in its own right, or $\Gamma$ is central. The latter case is ruled out as it implies $\Lambda = G$ and $\mu$ was assumed to be non-atomic.

\end{proof}

In \cite{abert2012growth}  the classical Borel density theorem is invoked to show that a subgroup $\Lambda \le G$ such that $G/\Lambda$ admits a $G$-invariant probability measure must be discrete. However, this reasoning does not apply in the positive characteristic case where Lie algebra methods are not as useful. The theorem of Margulis quoted in the proof overcomes this difficulty.

\section{The non simply-connected case and the real case}
\label{sec:non simply connected and the real case}

Up to this point the group $\GG$ was assumed to be simply-connected. We now discuss the more general situation of a possibly non simply-connected group.
This greater generality turns out to be necessary to be able to obtain the real Lie group version of Theorems \ref{thm:Intermediate-factor-theorem} and \ref{thm:stuck-zimmer-over-local-fields} as a formal corollary of the above discussion. 

However, without the simple-connectedness some arguments become technically more involved. For this reason we preferred to present the core ideas of the proofs in Sections \ref{sec:proof of IFT} and \ref{sec:The-Stuck-Zimmer-theorem} with the added assumption of simple-connectedness. 

In Subsection \ref{sub:statement of the theorems} we state our two main theorems in the non simply-connected case. The proofs are deferred to Subsection \ref{sub:proofs in non-simply connected case}. The real Lie group case in discussed in Subsection \ref{sub:the real case}.

Throughout the current section $\GG$ denotes a connected semisimple algebraic $k$-group without $k$-anisotropic almost $k$-simple subgroups\footnote{In other words, we drop the simple-connectedness assumption on $\GG$.}.

\subsection{Main theorems in the non-simply connected case}
\label{sub:statement of the theorems}

The statements will depend on a certain normal subgroup $\GG^+_k \lhd \GG_k$ that is defined and discussed in Subsection \ref{sub:the group G+} below. Note that $\GG_k = \GG^+_k$ whenever $\GG$ is simply-connected.

To state both Theorems \ref{thm:IFT-general form} and \ref{thm:SZ - general form}  we introduce the notation $G^\dagger$. Here  $G^\dagger$ is a subgroup satisfying
$$ \GG_k^+ \le G^\dagger \le \GG_k $$

\begin{thm}
	\label{thm:IFT-general form}
	Assume that $\rank{\GG}k\ge 2$. 
    Given an 	irreducible $G^\dagger$-space $X$ with finite invariant measure,  an
	ergodic $G^\dagger$-space $Y$ and  a pair of $G^\dagger$-maps
	$$X\times\nicefrac{\GG_k}{\PP_{k}}\to Y\to X$$ whose composition is the
	projection, the conclusion of Theorem \ref{thm:Intermediate-factor-theorem} holds true.
\end{thm}

\begin{thm}
	\label{thm:SZ - general form}
	Assume that $\rank{\GG}k\ge 2$
	and that $G^\dagger$ has property (T). Then every faithful, properly ergodic and irreducible $G^\dagger$-space with finite invariant measure is essentially free.
\end{thm}

Proofs of both theorems are given below in Subsection \ref{sub:proofs in non-simply connected case}.

\subsection{The group $\GG_k^+$}
\label{sub:the group G+}

Several additional subtleties that arise in the non-simply connected case have to do with the group $\GG_k ^+$.

\begin{defn}
\label{def: the group G+}
$\GG_{k}^{+}$ is the subgroup of $\GG_{k}$ generated by the $\mathbb{U}_k$ where $\mathbb{U}$ runs through the unipotent $k$-split subgroups of $\GG$.
\end{defn}

It is clear that $\GG_k^+$ is a normal subgroup. If $k$ is perfect\footnote{A local field $k$ is perfect if and only if $\mathrm{char}(k) = 0$.} then equivalently $\GG_k^+$ is  the subgroup generated by the unipotent elements of $\GG_k$.  As already noted above, $\GG_k^+ = \GG_k$ whenever $\GG$ is simply connected and has no $k$-anisotropic factors. These as well as the other properties of $\GG_{k}^{+}$ used below are discussed in sections I.1.5 and I.2.3 of \citep{margulis1991discrete}. 

\subsubsection*{$\GG^+_k$ and property (T)} 
It is shown in \cite{margulis1991discrete}, Lemma III.5.2 that the group $\GG_k$ has property (T) if and only if $\tilde{\GG}_k$ does, where $\tilde{\GG}$ denotes the algebraic universal cover group of $\GG$. Additionally for a local field $k$ the group $\GG_k^+$ is cocompact (and normal) in $\GG_k$. As is well-known (see e.g. \cite{margulis1991discrete}, Corollary III.2.13) in this situation $\GG_k$ has  (T) if and and only if $G^\dagger$ does.

To summarise, the necessary and sufficient conditions for $\GG_k$ to have property (T) that are given in the introduction (see Subsection \ref{sub:intro_nevo_stuck_zimmer_theorem}) apply equally well to the groups of the form $G^\dagger$.

\subsubsection*{$\GG^+_k$ and parabolic subgroups}

It is important for our purposes to clarify the relationship of the subgroup $\GG_k^+$ with the parabolic $k$-subgroups of $\GG$. 

Recall is $\SS$ a maximal $k$-split torus of $\GG$ and $\PP$ is a minimal $k$-parabolic subgroup (see Subsection \ref{sub:subgroup_structure}). We have that
$$ \GG_k = \GG_k^+ \cdot Z_\GG(\SS)_k$$
where $Z$ denotes the centralizer. On the other hand $Z_\GG(\SS) \le \PP$. This implies that
$$ \GG_k / \PP_k =  G / P = G^\dagger / (P \cap G^\dagger) $$
as well as the analogous fact with $\PP$ replaced by any other parabolic $\PP_\theta, \theta \subset \Delta$. In particular we have the following isomorphisms on the level of measure algebras
$$ \malg{\overline{V}} \cong \malg{G/P} \cong \malg{G^\dagger / (P \cap G^\dagger)} $$

To conclude the present discussion of parabolic subgroups we quote \cite{margulis1978quotient}, 1.13.

\begin{prop}
\label{prop:invariant under G+}
Let $\malg{}$ be a measure subalgebra of $\malg{G/P}$ that is invariant under $\GG_k^+$. Then $\malg{} = \malg{G/P_\theta}$ for some subset $\theta \subset \Delta$ of simple roots.
\end{prop}

Note that  we had already used this fact in the proof of Subsection \ref{sub:proof of ift}.

\subsubsection*{Mautner's lemma revisited}

The Mautner lemma for a non-simply connected group has a weaker formulation that takes into account the subgroup $\GG_k^+$ (see \citep{margulis1991discrete}, II.3.3). As in Subsection \ref{sub:statement of the theorems} above we let $G^\dagger$ denote any subgroup satisfying $\GG_k^+ \le G^\dagger \le \GG_k$.

\begin{lem}
	\label{lem:cor of Mautner lemma rel} 
	Let $\rho$ be a unitary representation of $G^\dagger$ into a Hilbert space $\mathcal{H}$.  Assume that $s\in S$ is
	such that $\text{Inn}\left(s\right)$ is contracting on some
	$V_{\theta}$ with $\theta \subsetneq \Delta$. Then every $v \in \mathcal{H}$ with $\rho(s)v = v$ satisfies $\rho(\HH_k^+) v = v$, where $\HH\le \GG$ is a product of a non-empty collection (depending on $s$) of the almost $k$-simple subgroups of $\GG$.
\end{lem}

In particular if $\GG$ is almost $k$-simple then the conclusion reads $\rho(\GG_k^+)v = v$. As usual, Lemma \ref{lem:cor of Mautner lemma rel} can be applied to the unitary representation associated to any $G^\dagger$-space with finite invariant measure.

\subsection{Real semisimple Lie groups}
\label{sub:the real case}

Recall that the group of $\R$-points $\GG_\R$ of any algebraic group $\GG$ defined over $\R$ is a Lie group. It is known that $\GG_\R$ has finitely many connected components (see \cite{platonov1994algebraic}, Theorem 3.6). 
However even if $\GG$ is connected as an algebraic group it could be that $\GG_\R$ is not topologically connected. 

For $\GG$ connected semisimple and without $\R$-anisotropic factors it turns out that the connected component at the identity $\GG_\R^0$ equals $\GG_\R^+$ (see \cite{margulis1991discrete}, I.2.3.1). In particular, if $\GG$ is algebraically simply-connected then $\GG_\R$ is topologically connected.

It is possible to go in the other direction as well; namely every connected
semisimple real Lie group $G$ with trivial center is the
$\GG_{\R}^{0}$ for some connected semisimple algebraic $\R$-group $\GG$ (see 3.1.6 in \citep{zimmer1984ergodic}). An almost simple factor of $G$ is compact
if and only if the corresponding factor of $\GG$ is $k$-anisotropic (see e.g. I.2.3.6
of \citep{margulis1991discrete}). There is no guarantee, however, that this $\GG$ can be chosen to be simply-connected.

Coming back to the Nevo-Zimmer and Stuck-Zimmer theorems, we see that the real Lie groups versions (namely theorems 4.1 of \cite{zimmer1982ergodic} and 2.1 of \cite{stuck1994stabilizers}) follow immediately from our Theorems \ref{thm:IFT-general form} and \ref{thm:SZ - general form} respectively, as long as the Lie group $G$ in question can be regarded as $\GG_\R^0$ for some corresponding algebraic $\R$-group. As mentioned above this is the case in particular if $G$ has trivial center, or more generally admits a faithful linear representation (see \cite{milne2011algebraic}, Theorem III.2.23).

We remark that as explained in \cite{stuck1994stabilizers} the Stuck-Zimmer theorem for real Lie groups can be reduced to the center-free case.

\subsection{Proofs in the non-simply connected case}
\label{sub:proofs in non-simply connected case}

We  explain the modifications in the proofs of the simple-connected counterparts  Theorems \ref{thm:Intermediate-factor-theorem} and \ref{thm:stuck-zimmer-over-local-fields} that are required to complete the proofs of Theorems \ref{thm:IFT-general form} and \ref{thm:SZ - general form}, respectively. 

Recall that in both theorems $G^\dagger$ is a subgroup with $\GG_k^+ \le G^\dagger \le \GG_k$.

\begin{proof}[Proof of Theorem \ref{thm:IFT-general form}]

A key ingredient in the proof of Theorem \ref{thm:Intermediate-factor-theorem} was the Main Lemma, proved in Subsection \ref{sub:proof_the main lemma}. We claim that in the present situation an entirely analogous statement to Lemma \ref{lem:main lemma} remains true, provided only that  we restrict ourselves to elements $g \in \GG_k^+$.

Indeed, the proof of the main lemma relies on exhibiting an element $g$ as a limit of certain elements $h_n u s^{-m_n}$. This in turn uses Proposition \ref{prop:Z_F is conull} to show that certain subsets $Z_N$ of $G \times X$ are conull. Altering the definition of $Z_N$ to
$$
Z_{N} = \left \{ \left(x,g\right)\in X\times G^\dagger \: : \:  \overline{\left\{ hgs^{-n}\right\} } = G^+, \; \text{where $n\ge 0, h\in G^\dagger$ and $ hx \in N$} \right\} 
$$
it is rather straightforward to see that the relevant variant of Mautner's lemma \ref{lem:cor of Mautner lemma rel} implies that $Z_N$ is conull whenever $\eta(N) > 0$; compare Proposition \ref{pro:GxZ acts ergodically} where the Mautner's lemma was used.

Having obtained the Main Lemma under the restriction that $g \in \GG_k^+$ we next consider the proof of Theorem \ref{thm:Intermediate-factor-theorem} as in Subsection \ref{sub:proof of ift}. First, in light of the above discussion regarding parabolic subgroups we may identify  the three measure algebras $ \malg{\overline{V}}$, $\malg{G/P}$  and $\malg{G^\dagger / (P \cap G^\dagger)} $. One further difficulty is 
that now the main lemma guarantees $G^+$-invariance only. Therefore we need to apply Proposition \ref{prop:invariant under G+} to deduce as before that such a $G^+$-invariant measure sub-algebra of $\malg{G/P}$ must be of the form $\malg{G/P_\theta}$ for a collection of simple roots $\theta \subset \Delta$.

Observe that except for these modifications the arguments of the proof go through.
\end{proof}

\begin{proof}[Proof of Theorem \ref{thm:SZ - general form}]

The argument showing that the action of $G^\dagger$ on $X$ can be assumed to be non weakly amenable (see Lemma \ref{lem:lemma on weak amenable actions}) is applicable to all  second countable locally compact groups and need not be modified.

The main part of the proof consists of applying the intermediate factor theorem. Recall that $G^\dagger/(P \cap G^\dagger) = G/P$  so that $G/P$ can be regarded as a $G^\dagger$-space and  Theorem \ref{thm:IFT-general form} can be applied exactly as in Section \ref{sec:The-Stuck-Zimmer-theorem} above. Let $\QQ$ be the proper parabolic $k$-subgroup of $\GG$ containing $\PP$ that is obtained from Theorem \ref{thm:IFT-general form}.

We deduce that $\mu$-almost every stabilizer $G^\dagger_x $ for the $G^\dagger$ action on $X$ satisfies
$$ G^\dagger_x \le H_x = \bigcap_{gQ^\dagger} g Q^\dagger g^{-1} $$
where we denote $Q^\dagger = \QQ_k \cap G^\dagger$ and the intersection is taken over almost all cosets $gQ^\dagger$ in $G^\dagger/Q^\dagger = G / Q$. Once more we see that $\mu$-almost every stabilizer $G_x^\dagger$ is  contained in a proper normal subgroup $H_x$ of $G^\dagger$. It remains to show that $H_x$ is centralised by some other normal subgroup so that Lemma \ref{lem:on stabilisers in proper normal subgroups} can be applied.

Since $\QQ$ is proper there exists   an almost $k$-simple subgroup $\HH$ of $\GG$ such that $\QQ \cap \HH \lneq \HH$. In particular
$\HH_k^+ \cap Q^\dagger $ and hence $\HH_k^+ \cap H_x $ are both proper subgroups of $\HH_k^+$. As every subgroup of $\HH_k$ that is normalised by $\HH_k^+$ is either central or contains $\HH_k^+$, the intersection $\HH_k^+ \cap H_x$ must be central. However taking into account that $G^\dagger$ is the almost direct product of the $\HH_k \cap G^\dagger$ where $\HH$ runs over the almost $k$-simple subgroups of $\GG$ we obtain $H_x \le Z_G(\HH_k)$ as required.
	
\end{proof}

\bibliographystyle{plainnat}
\bibliography{corrections}

\end{document}